\documentclass[preprint,12pt]{elsarticle}
\usepackage{lipsum}
\makeatletter
\def\ps@pprintTitle{
 \let\@evenfoot\@oddfoot}
\makeatother
\usepackage[T1]{fontenc}
\usepackage{lipsum}
\usepackage{algorithm}
\usepackage{amsmath}
\usepackage{algpseudocode}
\usepackage[tight,footnotesize]{subfigure}
\usepackage{flexisym}
\usepackage{breqn}
\newtheorem{theorem}{\textbf{Theorem}}

\newtheorem{corollary}[theorem]{\textbf{Corollary}}

\newtheorem{lemma}{\textbf{Lemma}}

\newenvironment{proof}[1][Proof]{\noindent\textbf{#1.} }{\ \rule{0.5em}{0.5em}}
\begin{document}

\title{New Formulas for Semi-Primes. Testing, Counting and Identification of the $n^{th}$ and next Semi-Primes}

\author[label1]{Issam Kaddoura}
\ead{issam.kaddoura@liu.edu.lb}
\author[label2]{Samih Abdul-Nabi}
\ead{samih.abdulnabi@liu.edu.lb}
\author[label1]{Khadija Al-Akhrass}

\address[label1]{Department of Mathematics, school of arts and sciences}
\address[label2]{Department of computers and communications engineering, \\ Lebanese International University, Beirut, Lebanon}


\begin{abstract}

In this paper we give a new semiprimality test and we construct a new formula for $\pi ^{(2)}(N)$, the function that counts the number of
semiprimes not exceeding a given number $N$. We also present new formulas to identify the $n^{th}$ semiprime and the next semiprime
to a given number. 
The new formulas are based on the knowledge of the primes less than or equal to the cube roots 
of $N : P_{1}, \; P_{2}....P_{\pi \left( \sqrt[3]{N}\right) }\leq \sqrt[3]{N}$.

\end{abstract}
\maketitle

\section{Introduction}

Securing data remains a concern for every individual and every organization on the globe. In telecommunication, cryptography is one of the 
studies that permits the secure transfer of information \cite{IEEE_standard} over the Internet. Prime numbers have special properties that 
make them of fundamental importance in cryptography. The core of the Internet security is based on protocols, such as SSL and TSL 
\cite{rescorla2001ssl} released in 1994 and persist as the basis for securing different aspects of today's Internet \cite{clark2013sok}.

The Rivest-Shamir-Adleman encryption method \cite{rivest1978method}, released in 1978, uses asymmetric keys for exchanging data. 
A secret key  $S_{k}$ and a public key $P_{k}$ are generated by the recipient with the following property: A message enciphered by $P_{k}$
can only be deciphered by $S_{k}$ and vice versa. The public key is publicly transmitted to the sender and used to encipher data that only the 
recipient can decipher. RSA is based on generating two large prime numbers, say $P$ and $Q$ and its security is enforced by the fact that 
albeit the fact that the product of these two primes $n = P \times Q $ is published, it is of enormous difficulty to factorize $n$.

A semiprime or (2 almost prime) or ( pq number) is a natural number that is a product of 2 primes not necessary distinct.
The semiprime is either a square of prime or square free. Also the square of any prime number is a semiprime number.

Mathematicians have been interested in many aspect of the semiprime numbers. In \cite{ishmukhametov2014distrubution} authors
derive a probabilistic function $g(y)$ for a number $y$ to be semiprime and an asymptotic formula for counting $g(y)$ when $y$ is very large.
In \cite{doss2013approximation} authors are interested in factorizing semiprimes and use an approximation to $\pi(n)$ the function that counts 
the prime numbers $\leq n$.

While mathematicians have achieved many important results concerning distribution of prime numbers. Many are interested 
in semiprimes properties as to counting prime and semiprime numbers not exceeding a given number. 
From \cite{weisstein2003semiprime,conway2008counting,goldston2009small}, the formula for $\pi ^{(2)}(N)$ that counts the semiprime 
numbers not exceeding $N$ is given by (\ref{eq:spc}).
 
\begin{equation}
\label{eq:spc}
\pi ^{(2)}(N)=\sum_{i=1}^{\pi \left( \sqrt{N}\right) }\left[ \pi \left( \frac{x}{p_{i}}\right) -i+1\right]
\end{equation}
\noindent This formula is based on the primes P$_{1}$,P$_{2},....,$P$_{\pi \left( \sqrt{N} \right) }\leq \sqrt{N}$ .

Our contribution is of several folds. First, we present a formula to test the semiprimality of a given integer, this formula is used to build a 
new function $\pi ^{(2)}(N)$ that counts the semiprimes not exceeding a given integer $N$ using only  P$_{1}$,P$_{2},....$P$_{\pi \left( \sqrt[3]{N}\right) }\leq \sqrt[3]{N}$.
Second, we present an explicit formula that identify the $n^{th}$ semiprime number. And finally we give a formula that finds the next 
semiprime to any given number.
\section{Semiprimality Test}
\label{sec:spt}

With the same complexity ${O}(\sqrt{x})$ as the Sieve of Eratosthenes to test a primality of a given number $x$, we employ the
Euclidean Algorithm and the fact that every prime number greater than 3 has the form $6k \pm  1$ and without previous knowledge about any prime,
we can test the primality of $x \geq  8$ using the following procedure :

Define the following functions

\begin{equation}
T_{0}(x)=\left\lfloor \frac{1}{2}\left( \left\lceil \frac{x}{2} -\left\lfloor \frac{x}{2}\right\rfloor \right\rceil +\left\lceil \frac{x}{3} -\left\lfloor \frac{x}{3}\right\rfloor \right\rceil \right) \right\rfloor
\end{equation}

\begin{equation}
T_{1}(x)=\left\lfloor \frac{1}{\left\lceil \frac{\sqrt{x}}{6} \right\rceil }\sum_{k=1}^{\left\lceil \frac{\sqrt{x}}{6}\right\rceil}\left\lceil \frac{x}{6k-1}-\left\lfloor \frac{x}{6k-1}\right\rfloor \right\rceil \right\rfloor
\end{equation}

\begin{equation}
T_{2}(x)=\left\lfloor \frac{1}{\left\lceil \frac{\sqrt{x}}{6} \right\rceil }\sum_{k=1}^{\left\lceil \frac{\sqrt{x}}{6}\right\rceil }\left\lceil \frac{x}{6k+1}-\left\lfloor \frac{x}{6k+1}\right\rfloor
\right\rceil \right\rfloor
\end{equation}

\begin{equation}
T(x)=\left\lfloor \frac{T_{0}+T_{1}+T_{2}}{3}\right\rfloor
\end{equation}

where $\left\lfloor x\right\rfloor $ and $\left\lceil x\right\rceil $ are the floor and the ceiling functions of the real number $x$ respectively.

We have the following theorem which is analog to that appeared in \cite{kaddoura2012formula} with slight modification and the details 
of the proof are exactly the same .

\begin{theorem}
\label{th:1}
Given any positive integer $x > 7$, then
\begin{enumerate}
   \item $x$ is prime if and only if  $T(x) =1$
	 \item  $x$ is composite if and only if $T(x) = 0$
   \item For $x >7$  
				\begin{equation}
				      \label{eq:1}
							\pi (x)=4+\sum_{j=1}^{\left\lceil  \frac{x-7}{6} \right\rceil}  T(6j+7)  +\sum_{j=1}^{\left\lceil \frac{x-5}{6} \right\rceil } T(6j+5) 
				\end{equation}
\end{enumerate}
counts the number of primes not exceeding $x$.
\end{theorem}

Now we proof the following Lemma:
\begin{lemma}
\label{l:1}
If $N$ is a positive integer with at least 3 factors, then there exist a prime $p$ such that:
\begin{equation*}
				p\leq \sqrt[3]{N} \mbox{ and $p$ divides $N$}
\end{equation*}
\end{lemma}

\begin{proof}
If $N$ has at least 3 factors then it can be represented as : $N=a.b.c$
with the assumption  $1<a\leq b\leq c$, we deduce that $N\geq a^{3}$ or $a\leq \sqrt[3]{N}$ .
By the fundamental theorem of arithmetic, $\exists $ a prime number $p$ such that $p$ divides $a$.
That means $\ p\leq a\leq \sqrt[3]{N},$ but $p$ divides $a$ and $a$ divides $N$, hence $p$ divides $N$ with the property $p\leq \sqrt[3]{N}.$
\end{proof}

Lemma \ref{l:1} tells that, if $N$ is not divisible by any prime $p\leq \sqrt[3]{N}$, then $N$ has at most 2 prime factors, i.e., $N$ is prime or semiprime .
 Using the proposed primality test defined by $T(x)$ we construct the semiprimality test as follows:

For $x \geq  8$, define the functions $K_{1}(x)$ and $K_{2}(x)$ as follows:

\begin{eqnarray}
K_{1}(x) &=&\left\lfloor \frac{1}{\pi (\left\lfloor \sqrt[3]{x}\right\rfloor )}\sum_{i=1}^{\pi \left (\left\lfloor \sqrt[3]{x}\right \rfloor \right  ) }\ \left\lceil 
\left\lceil \frac{x}{p_{i}}\right\rceil -\frac{x}{p_{i}}\right\rceil
\right\rfloor \\
K_{2}(x) &=&\left\lceil \frac{1}{\pi (\left\lfloor \sqrt[3]{x}\right\rfloor ) }\sum_{i=1}^{\pi \left ( \left\lfloor \sqrt[3]{x}\right\rfloor \right ) }\left\lfloor 
\frac{x}{p_{i}}-\left\lceil \frac{x}{p_{i}}\right\rceil +1\right\rfloor T\left(
\frac{x}{p_{i}} \right ) \right\rceil
\end{eqnarray}

where $\pi (x)$ is the classical prime counting function presented in (\ref{eq:1}), $T(x)$ is the same as in Theorem \ref{th:1}. Obviously $T(x)$ is independent of any previous knowledge of the prime numbers .

\begin{lemma}
\label{l:2}
If $K_{1}(x)=0$,  then $x$ is divisible by some prime $p_i \leq \sqrt[3]{x}$ .
\end{lemma}

\begin{proof}
For $K_{1}(x)=0$ ,we have $\left\lceil \left\lceil \frac{x}{p_{i}} \right\rceil -\frac{x}{p_{i}}\right\rceil =0$ \ for some $p_{i}$, then $x$ is
divisible by $p_{i}$ for some $p_{i}$ $\leq \sqrt[3]{x}$.
\end{proof}

\begin{lemma}
\label{l:3}
If $K_{1}(x)=1$ ,then $x$ has at most 2 prime factors exceeding $\sqrt[3]{x}.$
\end{lemma}

\begin{proof}
If $K_{1}(x)=1$, then $\left\lceil \left\lceil \frac{x}{p_{i}}\right\rceil - \frac{x}{p_{i}}\right\rceil =1$ for all $p_{i}$ $\leq \sqrt[3]{x}$ therefore
by lemma \ref{l:1}, $x$ is not divisible by any prime $p_{i}$ $\leq \sqrt[3]{x},$ therefore $x$ has at most two prime factors exceeding $\sqrt[3]{x}.$
\end{proof}

\begin{lemma}
If $T(x)$ $=0$ and $K_{1}(x)$ $=1,$ then x is semiprime and $K_{2}(x)$ $=0.$
\end{lemma}

\begin{proof}
If $K_{1}(x)=1$ , then $x$ has at most 2 prime factors but $T(x)=0$ which means that $x$ is composite,
hence $x$ has exactly two prime factors and both factors are greater than $\ \sqrt[3]{x}$ and 
$\left\lfloor \frac{x}{p_{i}}-\left\lceil \frac{x}{p_{i}}\right\rceil +1\right\rfloor =0$ for each prime $p_{i} \leq \sqrt[3]{x}$,
therefore $K_{2}(x)=0$.
\end{proof}

\begin{lemma}
If $T(x)=0$ and $K_{1}(x)=0$ , then $x$ is a semiprime number if and only if $K_{2}(x)=1$.
\end{lemma}

\begin{proof}
If $T(x)=0$ and $K_{1}(x)=0$ then $x$ divides a prime  $p \leq \sqrt[3]{N}$, but $x$ is semiprime that means 
$x =pq$ and $q$ is prime number hence for prime $p_{i}=p$ and $x=pq$ we have:
\begin{equation*}
\left\lfloor \frac{x}{p_{i}}-\left\lceil \frac{x}{p_{i}}\right\rceil
+1\right\rfloor T \left( \left\lceil \frac{x}{p_{i}} \right\rceil \right)=\left\lfloor 
\frac{pq}{p}-\left\lceil \frac{pq}{p}\right\rceil +1\right\rfloor
T\left( \left\lceil \frac{pq}{p}\right\rceil \right)=1
\end{equation*}
consequently $K_2(x) = 1$ because at least one of the terms is not zero.

conversely , if $K_{2}(x)=1$ then 
$\left\lfloor \frac{x}{p_{i}}-\left\lceil \frac{x}{p_{i}}\right\rceil +1\right\rfloor T \left( \left\lceil \frac{x}{p_{i}} \right\rceil \right)$ 
is not zero for some $i$ and then $x = p_{i}q$ and
$T \left( \left\lceil \frac{x}{p_{i}}\right\rceil \right)=1$
for some prime $p_{i}\leq \sqrt[3]{x}$ \ then $T \left( \left\lceil \frac{p_{i}q}{p_{i}} \right\rceil \right) =T(q)=1$ 
hence $q$ is a prime number and $x$ is a semiprime number .
\end{proof}

We are now in a position to prove the following theorem that characterize the semiprime numbers.

\begin{theorem}
\label{th,2}
( Semiprimality Test ): Given any positive integer  $x >7$, then $x$ is semiprime if and only if:
\begin{enumerate}
		\item $T(x)=0$ and $K_{1}(x)=1$ \\
		or 
		\item $T(x)=0$, $K_{1}(x)=0$ and $K_{2}(x)=1$
\end{enumerate}
\end{theorem}

\begin{proof}
If $x$ is semiprime then $x =pq$ where $p$ and $q$ are two primes . 
If $p$ and $q$ both greater than $\sqrt[3]{x}$ then 
$T(x)=0$ 
and 
\begin{equation*}
K_{1}(pq)= \left\lfloor \frac{1}{\pi \left ( \left\lfloor \sqrt[3]{pq}\right\rfloor \right ) } 
\sum_{i=1}^{\pi \left ( \left\lfloor \sqrt[3]{pq}\right\rfloor \right ) }\ \left\lceil \left\lceil \frac{pq}{p_{i}} \right\rceil -\frac{pq}{p_{i}}\right\rceil \right\rfloor =
\left\lfloor \frac{\pi \left ( \left\lfloor \sqrt[3]{pq}\right\rfloor \right ) }{\pi \left ( \left\lfloor \sqrt[3]{pq} \right\rfloor \right ) } \right\rfloor=1
\end{equation*}

if $x=p\textprime q\textprime$ where $p\textprime$ and $q\textprime$ are two primes such that $p\textprime \leq \left\lfloor \sqrt[3]{x}\right\rfloor$ 
and $q\textprime > \left\lfloor \sqrt[3]{x}\right\rfloor$ then $T(x)=0$ and

\begin{equation*}
K_{1}(p\textprime q\textprime)=\left\lfloor \frac{1}{\pi \left ( \left\lfloor \sqrt[3]{p\textprime q\textprime} \right\rfloor \right ) }
\sum_{i=1}^{\pi \left ( \left\lfloor \sqrt[3]{p\textprime q\textprime }\right\rfloor \right ) }\ \left\lceil \left\lceil 
\frac{p\textprime q\textprime}{p\textprime} \right\rceil -\frac{p\textprime q\textprime}{p\textprime}\right\rceil \right\rfloor =0
\end{equation*}

because $\left\lceil \left\lceil \frac{p\textprime q\textprime}{p\textprime} \right\rceil -\frac{\textprime q\textprime}{p\textprime}\right\rceil =0$ 
and 

\begin{equation*}
K_{2}(p\textprime q\textprime)=\left\lceil \frac{1}{\pi \left ( \left\lfloor \sqrt[3]{p\textprime q\textprime} \right\rfloor \right ) }
\sum_{i=1}^{\pi \left ( \left\lfloor \sqrt[3]{p\textprime q\textprime} \right\rfloor \right ) }\left\lfloor
\frac{p\textprime q\textprime}{p\textprime}-\left\lceil \frac{p\textprime q\textprime}{p\textprime} \right\rceil +1\right\rfloor T(\frac{p\textprime q\textprime}{p\textprime})\right\rceil =1
\end{equation*}

because $\left\lfloor \frac{p\textprime q\textprime}{p\textprime}-\left\lceil 
\frac{p\textprime q\textprime}{p\textprime} \right\rceil +1\right\rfloor T(\frac{p\textprime q\textprime}{p\textprime})=
\left\lfloor q\textprime -q\textprime +1\right\rfloor T(q\textprime)=1.$

The converse can be proved by the same arguments.
\end{proof}

\begin{corollary}
A positive integer x $>7$ is semiprime if and only if $K_{1}(x)+K_{2}(x)-T(x)=1.$
\end{corollary}

\begin{proof}
A direct consequence of the previous theorem and lemmas.
\end{proof}

\section{Semiprime Counting Function}
\label{sec:spcf}

Notice that the triple $(T(x), K_{1}(x), K_{2}(x))$ have only the following 4 possible cases only:

\begin{description}
	\item[Case 1:] \ \ $(T(x),K_{1}(x),K_{2}(x))=(1,1,0)$ indicates that $x$ is prime number.
	\item[Case 2:] \ \  $(T(x),K_{1}(x),K_{2}(x))=(0,1,0)$ indicates that $x$ is semiprime in the form $x=pq$ where $p$ and $q$ are primes 
	such that $\left\lfloor \sqrt[3]{x}\right\rfloor < p \leq \left\lfloor \sqrt[2]{x}\right\rfloor $ and 
	$q \geq \left\lfloor \sqrt[2]{x}\right\rfloor$.
	\item[Case 3:] \ \ $(T(x),K_{1}(x),K_{2}(x))=(0,0,1)$ indicates that $x$ is semiprime in the form $x=pq$ where $p$ and $q$ are primes 
	such that $p\leq\left\lfloor \sqrt[3]{x}\right\rfloor $ and $q=\frac{x}{p}\geq \frac{x}{\sqrt[3]{x}} \geq \left\lfloor \sqrt[3]{x^{2}}\right\rfloor$.
	\item[Case 4:] \ \ $(T(x),K_{1}(x),K_{2}(x))=(0,0,0)$ indicates that $x$ has at least 3 prime factors .
\end{description}

Using the previous observations, lemmas as well as Theorem \ref{th,2} and corollary, we prove the following theorem that includes a function that 
counts all semiprimes not exceeding a given number N .

\begin{theorem}
For $N \geq 8$ then 
\begin{equation}
\pi ^{(2)}(N)=2+\sum_{x=8}^{N}(K_{1}(x)+K_{2}(x)-T(x))
\end{equation}
is a function that counts all semiprimes not exceeding $N$.
\end{theorem}

\section{$N^{th}$ Semiprime Formula}

The first few semiprimes in ascending order are $sp_{1}=4, \ sp_{2}=6, \ sp_{3}=9, \ sp_{4}=10, \ sp_{5}=14, \ sp_{6}=15, \  sp_{7}=21, \ etc$

We define the function $G(n,x)=\left\lfloor \frac{2n}{n+x+1}\right\rfloor $ where $n = 1, 2, 3... $ and  $x = 0, 1, 2, 3...$

clearly  
\begin{equation*}
G(n,x)=\left\lfloor \frac{2n}{n+x+1}\right\rfloor =\left\{ 
\begin{array}{cc}
1 & x<n \\ 
0 & x\geq n%
\end{array}%
\right. 
\end{equation*}

knowing that the bound of the $n^{th}$ prime is $P_{n}\leq 2n\log n$ \cite{robin1983estimation} , we can say that \\
the $n^{th}$  semiprime $sp_{n} \leq 2$ $P_{n}\leq 4n\log n$

\begin{theorem}
For  $x \geq  8$ and  $n > 2$, $sp_{n}$ the $n^{th}$ semiprime is given by the formula
\end{theorem}

\begin{equation*}
sp_{n}=8+\sum_{x=8}^{\left\lfloor 4n\ln n \right\rfloor }\left\lfloor \frac{2n%
}{n+1+\pi ^{(2)}(x)}\right\rfloor =8+\sum_{x=8}^{\left\lfloor 4n\ln n
\right\rfloor }\left\lfloor \frac{2n}{n+3+\sum \limits_{m=8}^{x} \left 
(K_{1}(m)+K_{2}(m) - T(m) \right )}
\right\rfloor 
\end{equation*}

The formula in full is given by :

\begin{equation*}
\resizebox{.9\hsize}{!}
{$
sp_{n}=8+\sum \limits_{x=8}^{\left\lfloor 4n\ln n\right\rfloor }\left\lfloor \frac{2n%
}{n+3+\sum \limits_{m=8}^{x} \left (\left\lfloor \frac{1}{\pi (\left\lfloor \sqrt[3]{m}%
\right\rfloor )}\sum \limits_{i=1}^{\pi (\left\lfloor \sqrt[3]{m}\right\rfloor )}\
\left\lceil \left\lceil \frac{m}{p_{i}}\right\rceil -\frac{m}{p_{i}}%
\right\rceil \right\rfloor +\left\lceil \frac{1}{\pi (\left\lfloor \sqrt[3]{m%
}\right\rfloor )}\sum \limits_{i=1}^{\pi (\left\lfloor \sqrt[3]{m}\right\rfloor
)}\left\lfloor \frac{m}{p_{i}}-\left\lceil \frac{m}{p_{i}}\right\rceil
+1\right\rfloor T(\frac{m}{p_{i}})\right\rceil -T(m) \right)}\right\rfloor
$}
\end{equation*}

where $T(m)$ is given by

\begin{dmath*}
T(m)=\left\lfloor \frac{T_{0}(m)+T_{1}(m)+T_{2}(m)}{3}\right\rfloor
=\left\lfloor \frac{1}{3} \left ( \left\lfloor \frac{1}{2}\left( \left\lceil \frac{m%
}{2}-\left\lfloor \frac{m}{2}\right\rfloor \right\rceil +\left\lceil \frac{m%
}{3}-\left\lfloor \frac{m}{3}\right\rfloor \right\rceil \right)
\right\rfloor +\left\lfloor \frac{1}{\left\lceil \frac{\sqrt{m}}{6}%
\right\rceil }\sum_{k=1}^{\left\lceil \frac{\sqrt{m}}{6}\right\rceil
}\left\lceil \frac{m}{6k-1}-\left\lfloor \frac{m}{6k-1}\right\rfloor
\right\rceil \right\rfloor +\left\lfloor \frac{1}{\left\lceil \frac{\sqrt{m}%
}{6}\right\rceil }\sum_{k=1}^{\left\lceil \frac{\sqrt{m}}{6}\right\rceil
}\left\lceil \frac{m}{6k+1}-\left\lfloor \frac{m}{6k+1}\right\rfloor
\right\rceil \right\rfloor \right ) \right\rfloor
\end{dmath*}

\begin{proof}
For the $n^{th}$ semiprime $sp_{n}$, $\pi ^{(2)}(sp_{n})=n$ and for $x< sp_{i}$,  $pi^{(2)}(x) < pi^{(2)}(sp_i) = i \\ \forall \; i = 1 , 2 , 3 , ...., n$. 

Using the properties of the function $G(n,x)=\left\lfloor \frac{2n}{n+x+1%
}\right\rfloor =\left\{ 
\begin{array}{cc}
1 & x<n \\ 
0 & x\geq n%
\end{array}%
\right. $
\end{proof}

we compute

\begin{dmath*}
8+\sum_{x=8}^{\left\lfloor 4n\ln \right\rfloor }\left\lfloor \frac{2n}{%
n+1+\pi ^{(2)}(x)}\right\rfloor  =8+\sum_{x=8}^{\left\lfloor 4n\ln
\right\rfloor }G(n,\pi ^{(2)}(x))
=8 + G(n,\pi ^{(2)}(8))+G(n,\pi ^{(2)}(9))+G(n,\pi^{(2)}(10))+...+G(n,\pi ^{(2)}(P_{n-1}))+.... \\
+G(n,\pi ^{(2)}(P_{n-1}+1))+........G(n,\pi ^{(2)}(P_{n})+G(n,\pi^{(2)}(P_{n}+1)+...\\
=8+1+1+1+...1+0+0+0+...=sp_{n}
\end{dmath*}

where the last 1 in the summation is the value of $G(n,\pi ^{(2)}(sp_{n-1}))$
and then followed by $G(n,\pi ^{(2)}(sp_{n})=G(n,n)=0$ followed by zeros for
the rest terms of the summation, hence  
\begin{equation*}
sp_{n}=8+\sum_{x=8}^{\left\lfloor 4n\ln n\right\rfloor }G(n,\pi
^{2}(x))=8+\sum_{x=8}^{\left\lfloor 4n\ln n\right\rfloor }\left\lfloor \frac{%
2n}{n+1+\pi ^{(2)}(x)}\right\rfloor .
\end{equation*}

As an example, computing the $5^{th}$ semiprime number gives $sp_{5}=8+1+1+1+1+1+1 = 14$ as shown in Table \ref{table:5th}.
\begin{table}[h!]
\centering
\begin{tabular}{|r | r|} 
 \hline
$\pi^2(8)=2  $&$  G(5 ,\pi^2(8)) =  1  $    \\
$\pi^2(9)=3  $&$  G(5 ,\pi^2(9)) =  1   $   \\
$\pi^2(10)=4$&$  G(5 ,\pi^2(10)) =  1$    \\
$\pi^2(11)=4$&$ G(5 ,\pi^2(11)) =  1 $   \\
$\pi^2(12)=4$&$ G(5 ,\pi^2(12)) =  1$    \\
$\pi^2(13)=4$&$ G(5 ,\pi^2(13)) =  1$    \\
$\pi^2(14)=5$&$ G(5 ,\pi^2(14)) =  0$   \\\hline
\end{tabular}
\caption{Computing the $5^{th}$ semiprime}
\label{table:5th}
\end{table}

\section{Next Semiprine}

In our previous work \cite{kaddoura2012formula}, we introduced a formula that finds the next prime to a given number. 
In this section, we use an enhancement formula to find the next prime to a given number and we introduce a formula to compute the 
next semiprime to any given number. 


%
%
%
%
%
%

Recall that the integer $x\geq 8$ is a semiprime number if and only if $%
K_{1}(x)+K_{2}(x)-T(x)=1$ and if $x$ is not semiprime then $\
K_{1}(x)+K_{2}(x)-T(x)=0.$

Now we introduce an algorithm that computes the next semiprime to any given
positive integer $N$.

\begin{theorem}
If $N$ is any positive integer greater than 8 then the next semiprime to $N$
is given by: 
\begin{equation*}
NextSP(N)=N+1+\sum_{i=1}^{N}\left(
\prod\limits_{x=N+1}^{x=N+i} \left (1+T(x)-K_{1}(x)-K_{2}(x) \right)\right) 
\end{equation*}

where $T(x),K_{1}(x),K_{2}(x)$ are the functions defined in Section \ref{sec:spt}
\end{theorem}

\begin{proof}
We compute the summation:

\begin{dmath*}
\sum_{i=1}^{N}\left(
\prod\limits_{x=N+1}^{x=N+i}(1+T(x)-K_{1}(x)-K_{2}(x))\right)  \\
=\sum_{i=1}^{NextSP(N)-N-1}\left(
\prod\limits_{x=N+1}^{x=N+i}(1+T(x)-K_{1}(x)-K_{2}(x))\right)
+\sum_{i=NextSP(N)-N}^{N}\left(
\prod\limits_{x=N+1}^{x=N+i}(1+T(x)-K_{1}(x)-K_{2}(x))\right)  \\
=\sum_{i=1}^{NextP(N)-N-1}(1)+\sum_{i=NextP(N)-N}^{N}(0)=NextSP(N)-N-1
\end{dmath*}

hence 
\begin{equation*}
NextSP(N)=N+1+\sum_{i=1}^{N}\left(
\prod\limits_{x=N+1}^{x=N+i}(1+T(x)-K_{1}(x)-K_{2}(x))\right) 
\end{equation*}
\end{proof}

\begin{table}[ht!]
\centering
\begin{tabular}{|r r r|} 
 \hline
$x$ & $\pi^2(x)$ & Time in seconds \\ 
 \hline
\hline
10										& 4                 		&  0.00 \\
100									& 34									&  0.01 \\
1000								&	299								&  0.1  \\
10000							& 2625							&  3.0  \\	
100000      		& 23378       		&  50\\ 
1000000    		& 210035     		& 1091\\
10000000   		& 1904324   		&  22333\\
100000000 		& 17427258 		& 508840\\
\hline
\end{tabular}
\caption{Testing on $\pi^{(2)}(x)$}
\label{table:pie2(n)}
\end{table}

\section{Results}

We implemented the proposed functions using MATLAB and complete the testing on  an Intel Core i7-6700K with 8M cache and a clock speed of 4.0GHz. 						
Table \ref{table:pie2(n)}  shows the results related to $\pi^2(x)$ for some selected values of $x$. 

We have also computed few $n^{th}$ semiprimes as shown in Table \ref{table:nthsp}.

\begin{table}[ht!]
\centering
\begin{tabular}{|r r r|} 
 \hline
$n$ & $sp_n$ & Time in seconds \\ 
 \hline
\hline
100			            & 314       &  0.07\\ 
200									&	669      &  0.24 \\
300                  & 1003        & 0.49             \\
400                  &  1355       &   0.86           \\
500                  &    1735     &   1.22           \\
600                  &   2098      &  1.89            \\
700                  & 2474       &   2.39          \\
800                  &  2866       &  3.40           \\
900                  &  3202       &   3.78           \\
1000               & 3595     &  4.91\\
5000               &  19643  &   105.72        \\
10000     				&  40882  &  579.01\\
\hline
\end{tabular}
\caption{Testing on $n^{th}$ semiprimes}
\label{table:nthsp}
\end{table}

And finally we show the next semiprimes to some selected integers in Table \ref{table:nextsp}.

\begin{table}[ht!]
\centering
\begin{tabular}{|r r r|} 
 \hline
$n$ & $NextSP(n)$ & Time in seconds \\ 
 \hline
\hline
100			            & 106       &  0.01\\ 
200									&	 201     &  0.02 \\
300                  & 301        & 0.04             \\
400                  &  403      &   0.07           \\
500                  &    501     &   0.09           \\
1000               &   1003     &  0.31            \\
5000                  & 5001      &   5.92         \\
10000               & 10001    &  22.38\\
\hline
\end{tabular}
\caption{Testing on $NextSP(n)$ semiprimes}
\label{table:nextsp}
\end{table}

\section{Conclusion}
In this work, we presented new formulas for semiprimes. First, $\pi^{(2)}(n)$ that counts the number of semiprimes not exceeding a given number $n$.
Our proposed formula requires knowing only the primes that are less or equal $\sqrt[3]{n}$ while existing 
formulas require  at least knowing the primes that are less or equal $\sqrt[2]{n}$.
We also present a new formulas to identify the $n^{th}$ semiprime and finally, a new formula that gives the next semiprime to any integer. 

\section{References}
 \bibliographystyle{IEEEtran}
 \bibliography{IEEEabrv,paper}
 
  \end{document}